\documentclass[a4paper,12pt]{article}
\usepackage{latexsym}
\usepackage{amsmath}
\usepackage{amssymb}
\usepackage{enumerate}
\usepackage{pict2e}
\usepackage{theorem}
\newtheorem{theorem}{Theorem}[section]
\newtheorem{proposition}[theorem]{Proposition}
\newtheorem{lemma}[theorem]{Lemma}

\theorembodyfont{\rmfamily}
\newtheorem{proof}{\textmd{\textit{Proof.}}}

\newtheorem{remark}[theorem]{Remark}

\makeatletter

\@addtoreset{equation}{section}
\makeatother

\newcommand{\qedd}{\hfill \Box}

\newcommand{\R}{\ensuremath{\mathbb{R}}}
\newcommand{\Sph}{\ensuremath{\mathbb{S}}}

\def\g{\gamma}
\def\e{\varepsilon}

\title{Topology of complete Finsler manifolds admitting convex functions
\footnote{
Mathematics Subject Classification (2010)\,:\,53C60, 53C22.}
\footnote{
Keywords: 
}
}
\author{Sorin V. SABAU, Katsuhiro SHIOHAMA}

\date{}
\begin{document}


\maketitle

\section{Introduction.}
Let $(M,F)$ be an $n$-dimensional Finsler manifold. The well known Hopf-Rinow theorem (see for example \cite{BCS}) states that $M$ is complete if and only if the exponential map $\exp_p$ at some point $p\in M$ (and hence for every point on $M$) is defined on the whole tangent space $T_pM$ to $M$ at that point. This is equivalent to state that $(M,F)$ is geodesically complete with respect to forward geodesics at every point on $M$. Throughout this article we assume that {\it $(M,F)$ is geodesically complete with respect to forward geodesics}. 

A function $\varphi:(M,F)\to\R$ is said to be {\it convex} if and only if along every geodesic (forward and backward) $\g:[a,b]\to (M,F)$, the restriction $\varphi\circ\g:[a,b]\to\R$ is convex:  
   \begin{equation}\label{eq:convex}
       \varphi\circ\g((1-\lambda)a+\lambda\ b)\le(1-\lambda)\varphi\circ\g(a)+\lambda\varphi\circ\g(b),\quad 0\le\lambda\le 1.   
    \end{equation}   
    
If the inequality in the above relation is strict for all $\g$ and for all $\lambda\in(0,1)$, then $\varphi$ is called {\it strictly convex}. If the second order difference quotient, namely the quantity 
$\{\varphi\circ\g(h)-\varphi\circ\g(-h)-2\varphi\circ\g(0)\}\slash {h^2}$ is bounded away from zero on every compact set on $M$ along all $\g$, then $\varphi$ is called {\it strongly convex}. In the case when $\varphi$ is at least $C^2$, its convexity can be written in terms of the Finslerian Hessian of $\varphi$, but we do not need to do this in the present paper.

 If $\varphi\circ\gamma$ is a convex function of one variable, then the function 
 $\varphi\circ\bar{\gamma}$ is also convex, where $\bar\gamma$ is the reverse curve of $\gamma$. For a general Finsler metric if $\gamma$ is a geodesic it does not mean that the inverse curve $\bar\gamma$ is a geodesic also, but $\varphi\circ\bar\gamma$ is convex and so is 
 $\varphi\circ\gamma$ as well.

Every non-compact manifold admits a complete (Riemannian or Finslerian) metric and a non-trivial smooth function which is convex with respect to this metric (see \cite{GS1}).

 If a non-trivial convex function $\varphi:(M,F)\to\R$ is constant on an open set, then $\varphi$ assumes its minimum on this open set and the number of components of {\it a level set} $M^a_a(\varphi):=\varphi^{-1}(\{a\})$, $a\geq\inf_M\varphi$ is equal to that of the boundary components of the minimum set of $\varphi$. 
 Here we denote 
 $\inf_M\varphi:=\inf\{\varphi(x):x\in M\}$.
 
 A convex function $\varphi$ is said to be {\it locally non-constant} if it is not constant on any open set of $M$. From now on {\it we always assume that a convex function is locally non-constant}.
\par\medskip

The purpose of this article is to investigate the topology of complete Finsler manifolds admitting (locally non-constant) convex functions 
$\varphi:(M,F)\to\R$. Convex functions on complete Riemannian manifolds
have been fully discussed in~\cite{GS1} and others. Although the distance function on $(M,F)$ is not symmetric and the backward geodesics do not necessarily concide with the forward geodesics, we prove that most of the Riemannian results obtained in~\cite{GS1} have the Finsler extensions, as stated below.\par

 
 We first discuss the topology of the Finsler manifold $(M,F)$ admitting a convex function $\varphi$.

 \begin{theorem}[compare Theorem F,~\cite{GS1}]\label{Th:homeomorphism}
Let $\varphi:(M,F)\to\R$ be a convex function. Assume that all of the levels of $\varphi$ are compact.

If $\inf_M\varphi$ is not attained, then there exists a homeomorphism 
  \[    H:M^a_a(\varphi)\times(\inf_M\varphi,\infty)\to M,      \]
for an arbitrary fixed number $a \in (\inf_M\varphi , \infty)$,
such that 
   \begin{equation*}
\varphi(H(y,t))=t,\quad\forall y\in M^a_a(\varphi),\quad\forall t\in(\inf_M\varphi,\infty).     
   \end{equation*}
   
 Moreover, if $\lambda := \inf_M \varphi$ is attained, then $M$ is homeomorphic to the
normal bundle over $M_\lambda^\lambda (\varphi)$ in $M$.
 
 \end{theorem}

Next, we discuss the case where $\varphi$ has a disconnected level.

\begin {theorem}[compare Theorem A,~\cite{GS1}]\label{Th:disconnected} 
Let $\varphi:(M,F)\to\R$ be a convex function. If $M^c_c(\varphi)$ is disconnected for some $c\in\varphi(M)$, we then have
\begin{enumerate}[\rm(1)]
\item  $\inf_M\varphi$ is attained.
\item  If $\lambda:=\inf_M\varphi$, then $M^\lambda_\lambda(\varphi)$ is a totally geodesic smooth hypersurface which is totally convex without boundary. 
\item The normal bundle of $M^\lambda_\lambda(\varphi)$ in $M$ is trivial.
\item  If $b>\lambda$, then the boundary of the $b$-sublevel set 
$M^b(\varphi):=\{x\in M \,|\,\varphi(x)\le b\}$ has exactly two components.
\end{enumerate}
\end{theorem}

The diameter function $\delta:\varphi(M)\to\R_+$ plays an important role in this article and it is defined as follows:
     \begin{equation}\label{eq:diameterfunction}
        \delta(t):=\sup\{d(x,y) | x,y\in M_t^t(\varphi)\}.   
     \end{equation}   
     
It is  known from \cite{GS1} that the diameter function $\delta$ of a complete Riemannian manifold admitting a convex function is monotone non-decreasing. However it is not certain if it is monotone on a Finsler manifold. In Theorem \ref{Th:homeomorphism}, we do not use the monotone property but only 
the local Lipschitz property of $\delta$ which is proved in Proposition \ref{Prop:diameter}.

\bigskip

We finally discuss the number of ends of a Finsler manifold $(M,F)$ admitting a convex function $\varphi$. As stated above, the diameter function $\delta$, defined on the image of the convex function $\varphi$, may not be monotone. It might occur that a convex function defined on a Finsler manifold $(M,F)$ may simultaneously admit both compact and non-compact levels. This fact makes difficult to discuss the number of ends of the manifold $(M,F)$. However, we shall discuss all the possible cases and prove:

\begin{theorem}[compare Theorems C, D and G,~\cite{GS1}]\label{Th:ends}
Let $\varphi:(M,F)\to\R$ be a convex function. 
\begin{enumerate}
\renewcommand{\labelenumi}{\rm \Alph{enumi}.}
\item Assume that $\varphi$ admits a disconnected level. 
\begin{enumerate}[\rm ({A}1)]
\item  If all the level of $\varphi$ are compact, then $M$ has two ends.
\item If all the levels of $\varphi$ are non-compact, then $M$ has one end.
\item If both compact and non-compact levels of $\varphi$ exist 
simultaneously, then $M$ has at least three ends.
\end{enumerate}
\item Assume that all the levels of $\varphi$ are connected and compact. 
\begin{enumerate}[\rm ({B}1)]
\item If $\inf_M\varphi$ is attained, then $M$ has one end.
\item If $\inf_M \varphi$ is not attained, then $M$ has two ends.
\end{enumerate}
\item If all the levels are connected and non-compact, then $M$ has one
end.
\item Assume that all the levels of $\varphi$ are connected and that $\varphi$ admits
both compact and non-compact levels simultaneously. Then we have:
\begin{enumerate}[\rm ({D}1)]
\item If $\inf_M \varphi$ is not attained, then $M$ has two ends.
\item If $\inf_M \varphi$ is attained, then $M$ has at least two ends.
\end{enumerate}
\item Finally, if $M$ has two ends, then all the levels of $\varphi$ are compact.
\end{enumerate}

\end{theorem}

\begin{remark}
The supplementary condition that all of the levels of $\varphi$ are simultaneously compact or non-compact in the hypothesis of Theorem \ref{Th:homeomorphism} is necessary because we have not proved that the diameter function $\delta$ is monotone non-decreasing for a Finsler manifold. If this property of monotonicity would hold good, then this assumption can be removed. 
\end{remark}

 \medskip
We summarize the historical background of convex and related functions on manifolds, G-spaces and Alexandrov spaces. Locally non-constant convex functions, affine functions and peakless functions have been investigated on complete Riemannian manifolds and complete non-compact Busemann G-spaces and Alexandrov spaces in various ways. The topology of Riemannian manifolds $(N,g)$ admitting locally nonconstant convex functions, have been investigated in \cite{GS1}, \cite{Bangert}, \cite{GS2}, \cite{GS3}. The topology of Busemann G-surfaces admitting convex functions has been investigated in~\cite{Innami2} and in \cite{M1}. It should be noted that convex functions on complete Alexandrov surfaces are {\it not continous}. The notion of peakless functions introduced by Busemann~\cite{Busemann} is similar to quasiconvex functions and weaker than convex functions, and has been discussed in~\cite{Busemann-Phadke} and ~\cite{Innami3}. The topology of complete manifolds admitting locally geodesically (strictly) quasiconvex and uniformly locally convex filtrations have been investigated by Yamaguchi~\cite{Yamaguchi1},\cite{Yamaguchi2}
and \cite{Yamaguchi3}. The isometry groups of complete Riemannian manifolds $(N,g)$ admitting strictly convex functions have been discussed in \cite{Yamaguchi-isometry} and others. A well known classical theorem due to Cartan states that every compact isometry group on an Hadamard manifold $H$ has an fixed point. This follows from a simple fact that the distance function to every point on $H$ is strictly convex. Peakless functions and totally geodesic filtration on complete manifolds have been discussed in \cite{Innami3}, \cite{Busemann-Phadke}, \cite{Yamaguchi1}, \cite{Yamaguchi2}, \cite{Yamaguchi3} and others. 

A convex function on $(N,g)$ is said to be {\it affine} if and only if the equality in (\ref{eq:convex}) holds for all $\g$ and for all $\lambda\in(0,1)$. The splitting theorem for Riemannian manifolds have been investigated in ~\cite{Innami1}. Alexandrov spaces admitting affine functions have been established in \cite{Innami1}, \cite{M2} and~\cite{M3}. An overview on the convexity of Riemannian manifolds can be found in \cite{BZ}. \par
 The properties of isometry groups on Finsler manifolds admitting convex functions will be discussed separately. We refer the basic facts in Finsler and Riemannian geometry to \cite{BCS}, 
 \cite{CCL}, \cite{CE}, \cite{Sakai}. \par


  \section{Fundamental facts}
 
 The fundamental facts on convex sets and convex functions on $(M, F)$
are summarized as follows. Most of these are trivial in the Riemannian case, 
but we consider useful to
formulate and prove them in the more general Finslerian setting.

Let $(M,F)$ be a complete Finsler manifold. At each point $p\in M$, the 
indicatrix $\Sigma_p\subset T_pM$ at $p$ is defined as 
   \[    \Sigma_p:=\{u\in T_pM | F(p,u)=1\}.     \]
The {\it reversibility function} $\lambda:(M,F)\to\R^+$ of $(M,F)$ is given as
      \[     \lambda(p):=\sup\,\{F(p,-X)\,|\, X\in\Sigma_p\}.      \]
Clearly, $\lambda$ is continuous on $M$ and 
       \[   
    \lambda(p)=\max\{\frac{F(p,-X)}{F(p,X)}\,|\, X\in T_pM \setminus\{0\}\}.
       \]
Let $C\subset M$ be a compact set. There exists a constant $\lambda(C)>0$ depending on $C$ such that if $p\in C$ and if $X\in\Sigma_p$, then
   \[   \frac{1}{\lambda(C)}F(p,X)\le F(p,-X)\le\lambda(C)\cdot F(p,X).    \]
In particular, if $\sigma:[0,1]\to C$ is a smooth curve, then the length
$L(\sigma):=\int^1_0F(\sigma(t),\dot\sigma(t))dt$ of $\sigma$ satisfies
     \[  \frac{1}{\lambda(C)}L(\sigma)\le L(\sigma^{-1})\le\lambda(C)\cdot L(\sigma).   \]
Here we set $\sigma^{-1}(t):=\sigma(1-t)$, $t\in[0,1]$ the reverse curve of $\sigma$. 

It is well known
(see for example \cite{BCS})  that the topology of $(M,F)$ as an inner metric space is equivalent to that of $M$ as a manifold. For a compact set $C\subset M$, the inner metric $d_F$ of $(M,F)$ induced from the Finslerian fundamental function has the property:
   \[  \frac{1}{\lambda(C)}d_F(p,q)\le d_F(q,p)\le \lambda(C)\cdot d_F(p,q),\quad\forall p,q\in C.  \] 

%

 Let ${\rm inj}:(M,F)\to\R_+$ be the {\it injectivity radius function} of the exponential map. Namely, ${\rm inj}(p)$ for a point $p\in M$ is the maximal radius of the ball centred at the origin of the tangent space $T_pM$ at $p$ on which $\exp_p$ is injective. 
 
A classical result due to J. H. C. Whitehead~\cite{Whitehead} states that 
there exists a {\it convexity radius function} $r : (M, F ) \to \R$ such
that if $B(p, r) := \{x \in M | d(p, x) < r\}$ is an $r$-ball centered at $p$, then
$B(q, r' ) \subset B(p, r)$ for every $q \in B(p, r(p))$ and for every 
$r' \in (0, r(p))$ is
{\it strongly convex}.
 Namely, the distance function to $p$ is strongly convex along every geodesic in $B(p,r)$, $r\in(0,r(p))$, if its extension does not pass through $p$.

 A closed set $U \subset M$ is called {\it locally convex} if and only if
$U \cap B(p, r)$, for every $x \in U$ and for some $r \in (0, r(p))$, is convex. 
Pay attention to the fact that this definition has sense only for closed sets, since every open set is obviously locally convex.

 A set $V\subset M$ is called {\it totally convex} if and only if every geodesic joining two points on $V$ is contained entirely in $V$. A closed hemi-sphere in the standard sphere $\Sph^n$ is locally convex and an open hemi-sphere is strongly convex, while $\Sph^n$ itself is the only one totally convex set on it. 
 
 The minimum set of a convex function on $(M,F)$ is totally convex, if it exists.\par\medskip
 
 \begin{proposition} {\rm (see \cite{Sabau} ; Theorem 4.6)}. Let $C \subset M$ be a compact
set. Let $\lambda(C)$ be the reversibility constant of the compact Finsler space $(C,F)$. If $r(C)$ and ${\rm inj}_C$ are
the convexity and injectivity radii of $C$, respectively, we then have:
\begin{equation}
r(C) \leq
\frac{\lambda}{1+\lambda}{\rm inj}_C.
\end{equation}
\end{proposition}
\begin{proposition}\label{Prop:Lipschitz}
{\rm
A convex function $\varphi:(M,F)\to\R$ defined as in \eqref{eq:convex}
is locally Lipschitz.
}
\end{proposition}

\begin{proof}

Let $C\subset M$ be an arbitrary fixed compact set and $C_1:=\{x\in M\;|\; d(C,x)\le 1\}$. Here we set $d(C,x):=\min\{d(y,x)\,|\,y\in C\}$. For points $x, y\in C_1$ we denote by $\g_{xy}:[0,d(x,y)]\to M$, $\g_{yx}:[0,d(y,x)]\to M$ minimizing geodesics with $\g_{xy}(0)=x$, $\g_{xy}(d(x,y)=y$ and $\g_{yx}(0)=y$, $\g_{yx}(d(y,x))=x$. The slope inequalities along the convex functions: $\varphi\circ\g_{xy}|_{[0,d(x,y)+1]}$ and $\varphi\circ\g_{yx}|_{[0,d(y,x)+1]}$ imply that if $\Lambda:=\sup_{C_1}\varphi$ and $\lambda:=\inf_{C_1}\varphi$ (see Figure \ref{A convex function is locally Lipschitz}), then
  \[  \frac{\varphi(y)-\varphi(x)}{d(x,y)}\le\Lambda-\lambda,\quad  \frac{\varphi(x)-\varphi(y)}{d(y,x)}\le\Lambda-\lambda.    \]


\setlength{\unitlength}{1cm} 
 \begin{figure}[h]
    \begin{center}
  \begin{picture}(9, 9)
  \thicklines
  \put(1,1){\line(0,1){8}}
  \put(0,2){\line(1,0){10}}
  \put(0.5,1.5){$O$}
  \put(1.2,2.5){$\varphi\circ\g_{xy}$}
  \put(1.8,5){$\varphi\circ\g_{yx}$}
  \thinlines
  \multiput(1,4)(0.3,0){27}{\line(1,0){0.1}}
  \multiput(1,6)(0.3,0){27}{\line(1,0){0.1}}  
   \thicklines
  \qbezier(1,4)(3,0)(7.5,8)
  \qbezier(1,6)(5,0)(10,6)
  \thinlines
  \put(6.3,2){\line(0,1){4.03}}

  \put(7.3,2){\line(0,1){5.6}}

  \put(7.9,2){\line(0,1){2}}
  \put(8.9,2){\line(0,1){2.8}}

  \put(0,4){$\varphi(x)$}
  \put(0,6){$\varphi(y)$}
  \put(5.6,1.6){$d(x,y)$}
  \put(7.5,1.6){$d(y,x)$}
\put(6.7,2.2){$1$}
  \put(8.3,2.2){$1$}
  \qbezier(6.3,2)(6.4,2.1)(6.5,2.2)
  \qbezier(7.1,2.2)(7.2,2.1)(7.3,2)
  \qbezier(7.9,2)(8,2.1)(8.1,2.2)
  \qbezier(8.7,2.2)(8.8,2.1)(8.9,2)
  \end{picture}
  \caption{A convex function is locally Lipschitz.}\label{A convex function is locally Lipschitz}
    \end{center}
  \end{figure}
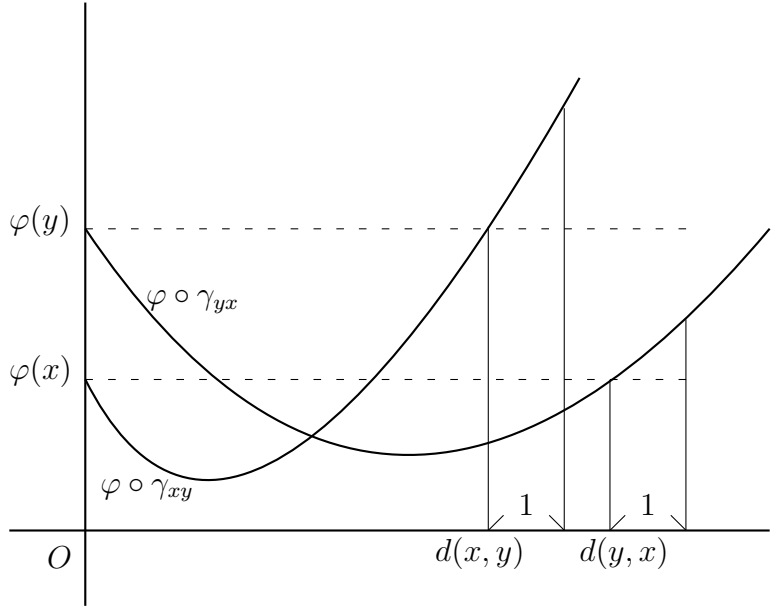
  
  \bigskip
  
There exists a constant $L=L(C)>0$ such that
   \[    \sup\{\frac{d(x,y)}{d(y,x)}\,|\, x,y\in C\}\le L.    \]

Therefore we have
   \[   |\frac{\varphi(x)-\varphi(y)}{d(x,y)}|,\; |\frac{\varphi(y)-\varphi(x)}{d(y,x)}|\le L(\Lambda-\lambda).     \]
$\qedd$
\end{proof}

\begin{proposition}\label{Prop:convex}
{\rm
If $C\subset (M,F)$ is a closed locally convex set, then there exists a $k$-dimensional totally geodesic submanifold $W$ of $M$ contained in $C$ and its closure coincides with $C$.
}
\end{proposition}

\begin{proof}
Let $r:(M,F)\to\R$ be the convexity radius function. For every point $p\in C$ there exists a $k(p)$-dimensional smooth submanifold of $M$ which is contained entirely in $C$ and such that $k(p)$ is the maximal dimension of all such submanifolds in $C$, where $0\le k(p)\le n$. At least $\{p\}$ is a $0$-dimensional such a submanifold contained in $C$. 

Let $K\subset M$ be a large compact set containing $p$ and $r(K)$ the convexity radius of $K$, namely $r(K):=\min\{ r(x)\,|\, x\in K\}$. Let $k:=\max\{ k(p)\,|\, p\in C\}$. 

Let $W(p)\subset C$ be a $k$-dimensional smooth submanifold of $M$. Suppose that $W(p)\cap B(p;r)\subsetneqq C\cap B(p;r)$ for a sufficiently small $r\in(0,r(K))$. Then, there exists a point $q\in B(p;r)\cap (C\setminus W(p))$. Clearly $\dot\g_{pq}(0)$ is transversal to $T_pW(p)$, and hence a family of minimizing geodesics $\{\g_{xq}:[0,d(x,q)]\to B(p;r)\,|\, x\in W(p)\cap B(p;r)\}$ with $\g_{xq}(0)=x$, $\g_{xq}(d(x,q))=q$ has the property that every $\dot{\g}_{xq}(0)$ is transversal to $T_xW(p)$. Therefore, this family of geodesics forms a $(k+1)$-dimensional submanifold contained in $C$, a contradiction to the choice of $k$. This proves $W(p)\cap B(p;r)=C\cap B(p;r)$ for a sufficiently small $r\in(0,r(K))$. We then observe that $\cup_{p\in C}W(p)=:W\subset C$ forms a $k$-dimensional smooth submanifold which is totally geodesic. Indeed, for any tangent vector $v$ to $W$, there exists $p\in C$ such that $v\in T_pW(p)$ and due to the convexity of $C$, the geodesic $\gamma_v:[0,\varepsilon]\to M$ cannot leave the submanifold $W$.

\setlength{\unitlength}{2cm} 
\begin{figure}[h]
    \begin{center}
\begin{picture}(6, 5)
\put(3,2){\circle*{0.05}}
\put(3,1.7){$x$}
\put(3,4){$B(x,r(K))$}
\put(2,4){$\bar{W}$}
\put(1.77,2){\circle*{0.05}}
\put(1.77,2){\circle{1}}
\put(1.67,2.2){$z$}
\put(1.5,1.3){$U\subset W$}

\put(2.48,2){\circle*{0.05}}

\put(2.6,2.2){$y$}
\qbezier(4.9,2.0)(4.9,2.787)(4.3435,3.3435) 
\qbezier(4.3435,3.3435)(3.787,3.9)(3.0,3.9) 
\qbezier(3.0,3.9)(2.213,3.9)(1.6565,3.3435) 
\qbezier(1.6565,3.3435)(1.1,2.787)(1.1,2.0) 
\qbezier(1.1,2.0)(1.1,1.213)(1.6565,0.6565) 
\qbezier(1.6565,0.6565)(2.213,0.1)(3.0,0.1) 
\qbezier(3.0,0.1)(3.787,0.1)(4.3435,0.6565) 
\qbezier(4.3435,0.6565)(4.9,1.213)(4.9,2.0)


\qbezier(1.8,5)(3.3,3.2)(1.8,0)
\qbezier(3,2)(2.5,2)(1.77,2)
\thicklines
\qbezier[40](3,2)(2.5,2.2)(1.85,2.35)

\qbezier[40](3,2)(2.5,2.1)(1.85,2.2)

\qbezier[40](3,2)(2.5,1.8)(1.85,1.65)

\end{picture}
\caption{The closure $\bar W$ of $W$ coincides with $C$.}\label{The closure}
    \end{center}
  \end{figure}
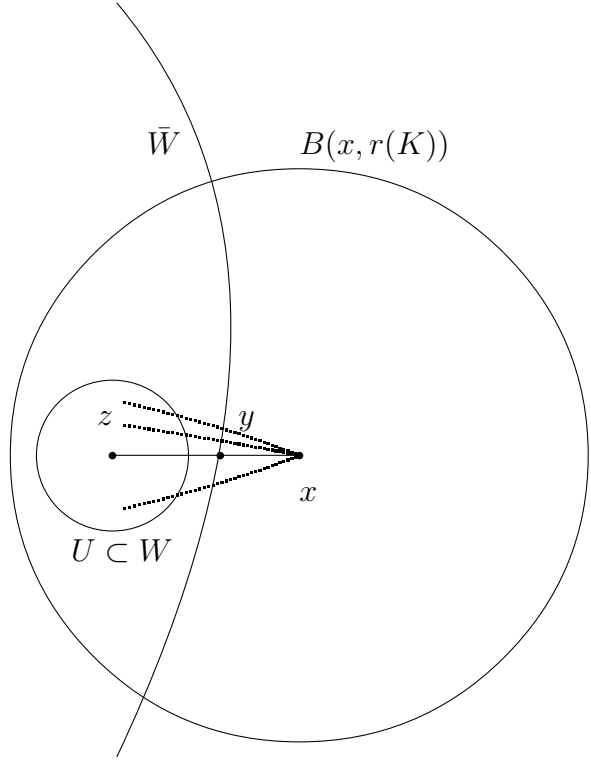



We finally prove that the closure $\bar{W}$ of $W$ coincides with $C$. Indeed, suppose that there exists a point $x\in C\setminus\bar W$. We then find a point $y\in\bar W\setminus W$ such that 
$d(x,y)=d(x,\bar W)<r(K)$. If $\dot\g_{xy}(d(x,y))\in T_y\bar W:=\lim_{y_j\to y}T_{y_j}M$, then 
$\gamma_{xy}(d(x,y)+\varepsilon)\in W$ for a sufficiently small $\varepsilon>0$. Let $U\subset W\cap B(x,r(K))$ be an open set around $\gamma_{xy}(d(x,y)+\varepsilon)$.

Then a family of geodesics
   \begin{equation}\label{eq:locallyconvex}
      \{\g_{xz}:[0,d(x,z)]\to B(x;r(K))\,|\, z\in U\}   
   \end{equation}   
forms a $k$-dimensional submanifold contained in $W$ and hence $y\in W$, 
a contradiction to $y\in \bar W\setminus W$.  Therefore, $\dot\g_{xy}(d(x,y))$ does not belong to $T_y\bar W$, and (\ref{eq:locallyconvex}) again forms a $(k+1)$-dimensional submanifold in $C$, a contradiction to the choice of $k$ (see Figure \ref{The closure}).
$\qedd$
\end{proof}

Let $C\subset M$ be a closed locally convex set and $p\in C$. There exists a totally geodesic submanifold $W\subset C$ as stated in Proposition \ref{Prop:convex}. We call $W$ the {\it interior of $C$} and denote it by ${\rm Int}(C)$. The {\it boundary of $C$} is defined by $\partial C:=C\setminus{\rm Int}(C)$, and {\it the dimension of $C$} is defined by $\dim C:=\dim{\rm Int}(C)$. The {\it tangent cone $\mathcal{C}_p(C)\subset T_pM$} of $C$ at a point $p\in C$ is defined as follows:
       \begin{equation}\label{def:tangent_cone}
   \mathcal{C}_p(C):=\{\xi\in T_pM\,|\, \exp_pt\xi\in{\rm Int}(C),\quad\text{for some $t>0$}\}.     
       \end{equation}

Clearly, $\mathcal{C}_p(C)=T_p{\rm Int}(C)\setminus\{0\}$ for $p\in{\rm Int}(C)$. 

We also define the tangent space $T_pC$ of $C$ at a point $p\in\partial C$ by $T_pC:=\lim_{q\to p}T_q{\rm  Int}(C)$. 

We claim that there exists for every point $p\in \partial C$ an open half space $T_pC_+\subset T_pC$ of $T_pC$ such that $\mathcal{C}_p(C)$ is contained entirely in an open half space $T_pC_+\subset T_pC$ :   
        \begin{equation}\label{eq:tangentcone}   
  \mathcal{C}_p(C)\subset T_pC_+ \subset T_pC:=\lim_{q\to p}\,T_q{\rm Int}(C),\; q\in{\rm Int}(C).     
        \end{equation}
        
Indeed, let $p\in\partial C$ and $\g_{qp}:[0,d(q,p)]\to B(p;r(K))$ for every point $q\in B(p;r(K))\cap {\rm Int}(C)$ be a minimizing geodesic. Suppose that there is a point $q\in {\rm Int}(C)$ such that $z:=\g_{qp}(d(q,p)+\e)\in C$ for a sufficiently small $\e>0$. 
We then have $\dot\gamma_{qp}(d(q,p))\in T_pC$, and hence the convex cone as obtained in 
 \eqref{eq:locallyconvex}
 is contained in $C$, a contradiction to the choice of $p\in \partial C$.
From the above argument we observe that if $p\in\partial C$, then there exists a hyperplane $H_p\subset T_pC$ such that $\mathcal{C}_p(C)$ is contained in a half space $T_p(C)_+\subset T_pC$ bounded by $H_p$.\par\medskip

\begin{proposition}\label{prop:level}
{\rm
Let $\varphi:(M,F)\to\R$ be a convex function. Then, $M^a_a(\varphi)$ for every $a>\inf_M\varphi$ is an embedded topological submanifold of dimension $n-1$.
}
\end{proposition}
\begin{proof}
Let $p\in M^a_a(\varphi)$ and $q\in B(p;r(p))\cap{\rm Int}(M^a(\varphi))$. There exists a hyperplane $H_p\subset T_pM$ such that 
        \[    
H_p=\partial T_p(M^a(\varphi))_+\quad\text{and}\quad \mathcal{C}_p(M^a(\varphi))\subset T_p(M^a(\varphi))_+.
        \]
Every point $x\in \exp_p(H_p)\cap B(p;r(p))$ is joined to $q$ by a unique minimizing geodesic $\g_{qx}:[0,d(q,x)]\to M$ such that $\g_{qx}(0)=q$, 
$\g_{qx}(d(q,x))=x$. Then there exists a unique parameter $t(x)\in(0,d(q,x)]$ such that $\g_{qx}(t(x))\in M^a_a(\varphi)\cap B(p;r(p))$. Let $B_H(O;r(p))$ be the open $r(p)$-ball in $H_p$ centered at the origin $O$ of $M_p$. We then have a map $\alpha_p:B_H(O;r(p))\to M^a_a(\varphi))$ such that
     \[    \alpha_p(u):=\g_{qx}(t(x)),\quad u\in B_H(O;r(p)),\quad \exp_pu=x.   \]

\setlength{\unitlength}{1.6cm} 
 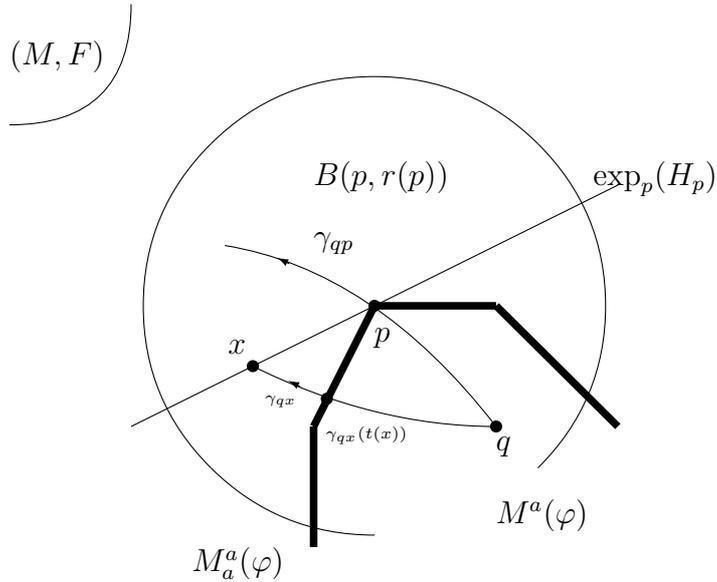
\begin{figure}[h]
    \begin{center}
\begin{picture}(6, 4)
\put(4.8,3){$\exp_p(H_p)$}
\put(2.5,3){$B(p,r(p))$}
\put(0,4){$(M,F)$}
\qbezier(1,4.5)(1,3.5)(0,3.5) 

\put(3,2){\circle*{0.1}}
\put(3,1.7){$p$}
\qbezier(4.9,2.0)(4.9,2.787)(4.3435,3.3435) 
\qbezier(4.3435,3.3435)(3.787,3.9)(3.0,3.9) 
\qbezier(3.0,3.9)(2.213,3.9)(1.6565,3.3435) 
\qbezier(1.6565,3.3435)(1.1,2.787)(1.1,2.0) 
\qbezier(1.1,2.0)(1.1,1.213)(1.6565,0.6565) 
\qbezier(1.6565,0.6565)(2.213,0.1)(3.0,0.1) 
\qbezier(4.3435,0.6565)(4.9,1.213)(4.9,2.0)

\qbezier(4,1)(3,2.3)(1.77,2.5)
\put(4,1){\circle*{0.1}}
\put(4,0.8){$q$}
\put(2.3,2.35){\vector(-2,1){0.1}}
\put(2.5,2.5){$\g_{qp}$}
\put(4,0.2){$M^{a}(\varphi)$}
\put(1,1){\line(2,1){4}}
\qbezier(2,1.5)(3,1)(4,1)
\put(2,1.5){\circle*{0.1}}
\put(1.8,1.6){$x$}
\put(1.5,-0.2){$M_{a}^{a}(\varphi)$}
\put(2.6,0.9){\tiny{$\g_{qx}(t(x))$}}
\put(2.61,1.23){\circle*{0.1}}
\put(2.1,1.2){\tiny{$\g_{qx}$}}

\put(2.39,1.33){\vector(-2,1){0.1}}

\linethickness{1mm} 

\put(3,2){\line(1,0){1}}
\put(4,2){\line(1,-1){1}}
\put(3,2){\line(-1,-2){0.5}}
\put(2.5,1){\line(0,-1){1}}
\end{picture}
\caption{An atlas of local charts at an arbitrary point $p\in M_a^a(\varphi)$.}\label{The open $r(K)$-ball}
    \end{center}
  \end{figure}



Clearly, $\alpha_p$ gives a homeomorphism between $B_H(O;r(p))$ and its image in $M^a_a(\varphi)$. Thus the family of maps $\{(B_H(O;r(p)),\alpha_p)\,|\, p\in M^a_a(\varphi)\}$ forms an atlas of $M^a_a(\varphi)$ (see Figure \ref{The open $r(K)$-ball}).

$\qedd$

\end{proof}


\section{Level sets configuration}

We shall give the proofs of Theorems \ref{Th:disconnected} and \ref{Th:ends}. The following
Lemma \ref{Lem:compactlevel} is elementary and useful for our discussion.

\begin{lemma}\label{Lem:compactlevel}
Let $\varphi:(M,F)\to \R$ be a convex function. If $M_{a}^{a}(\varphi)$ is compact, then so is $M_{b}^{b}(\varphi)$ for all $b\geq a$. If $M_{a}^{a}(\varphi)$ is non-compact, then so is $M_{b}^{b}(\varphi)$ for all $b\leq a$. 
\end{lemma}

\begin{proof}
First of all we prove that if $M_{a}^{a}(\varphi)$ is compact, then so is $M_{b}^{b}(\varphi)$ for all $b\geq a$.

Suppose that  $M_{b}^{b}(\varphi)$ is noncompact for some $b>a$. Take a point $p\in M_{a}^{a}(\varphi)$ and a divergent sequence $\{q_{j}\}$, $j=1,2,\dots$ on $M_{b}^{b}(\varphi)$. Since $M_{a}^{a}(\varphi)$ is compact, there is a positive number $L$ such that $d(p,x)<L$ for all $x\in M_{a}^{a}(\varphi)$. Let $\gamma_{j}:[0,d(p,q_{j})]\to M$ be a minimizing geodesic with $\gamma_{j}(0)=p$, 
$\gamma_{j}(d(p,q_{j}))=q$ for $j=1,2,\dots$. 
Compactness of $M_{a}^{a}(\varphi)$ implies that each $\varphi\circ\gamma_j |_{[L,d(p,qj )-L]}$
is monotone and non-decreasing for all large numbers $j$. 

Choosing a subsequence 
$\{\gamma_{i}\}$ of $\{\gamma_{j}\}$ if necessary, we find a ray $\gamma_{\infty}:[0,\infty)\to M$ emanating from $p$ such that $\varphi\circ\gamma_{\infty}$ is monotone, non-decreasing and bounded above, and hence is constant =$a$. This contradicts the assumption that $M_{a}^{a}(\varphi)$ is compact. 

$\qedd$
\end{proof}


The following Proposition \ref{Prop:homeomorphism} is the basic part of the proof of Theorem \ref{Th:homeomorphism}. Under the assumptions in Theorem \ref{Th:homeomorphism}, we divide $M$ into countable compact sets such that $M=\cup_{j=-\infty}^\infty\varphi^{-1}[t_{j-1},t_j]$, where $\{t_j\}$ is monotone increasing and $\lim_{j\to -\infty}t_j=\inf_M\varphi$ (if $\inf_M\varphi$ is not attained) and $\lim_{j\to \infty}t_j=\infty$. In the application of Proposition  \ref{Prop:homeomorphism} to each $\varphi^{-1}[t_{j-1},t_j]$, the undefined numbers $b_{k+1}$ and $b_0$, appearing in the proof of Proposition  \ref{Prop:homeomorphism}, play the role of margin to be pasted with $\varphi^{-1}[t_{j},t_{j+1}]$ (using $b_{k+1}$) and with 
$\varphi^{-1}[t_{j-2},t_{j-1}]$ (using $b_{-1}$), respectively.

\begin{proposition}\label{Prop:homeomorphism}
Let $M_a^a (\varphi) \subset M$ be a connected and compact level set 
and $b > a$ a fixed value. Then there exists a homeomorphism 
$\Phi^b_a:
M_b^b (\varphi) \times [a, b] \to M_a^b (\varphi)$ such that
\begin{equation}
\varphi\circ \Phi_a^b(x, t) = t,\ 
(x, t) \in M_b^b (\varphi) \times [a, b].
\end{equation}
\end{proposition}

\begin{proof}
Let $K\subset M$ be a compact set with $M^b_a(\varphi)\subset{\rm Int}(K)$ and $r:=r(K)$ the convexity radius over $K$. We define two divisions as follows. 
Let $a=a_0<a_1<\cdots<a_k=b$ and $b_{-1}<b_0<\cdots<b_k$ be given such that $\varphi^{-1}[b_{-1},b_k]\subset {\rm Int}(K)$ and
   \begin{enumerate}
 \item  $b_{-1}<a_0<b_0<a_1<\cdots <a_{k-1}<b_{k-1}<a_k=b<b_k$,   \\
 \item  $b_j:=\frac{a_j+a_{j+1}}{2},\quad j=0,1,\cdots,k-1$,   \\
 \item  $\varphi^{-1}(\{a_{j-1}\})\subset\bigcup\,\{B(x,r) \,|\, x\in\varphi^{-1}(\{a_{j+1}\})\}$,
  $\;j=1,\cdots,k-1$ 
  \item $\varphi^{-1}(\{b_{-1}\})\subset\bigcup\, \{B(y,r)\, |\, y\in\varphi^{-1}(\{a_1\})\}$,  \\
  \item  $\varphi^{-1}(\{a_{k-1}\})\subset\bigcup\, \{B(z,r)\, |\, z\in\varphi^{-1}(\{b_k\})\}$.
   \end{enumerate}
Obviously we have $[a,b]\subset [b_{-1},b_k]$.

For an arbitrary fixed point $p'_j\in\varphi^{-1}(\{a_{j+1}\})$, we have a minimizing geodesic $T(p'_j,q_j)$ realizing the distance $d(p'_j,\varphi^{-1}(-\infty,a_{j-1}])$ and $q_j$ the foot of $p'_j$ on $\varphi^{-1}(-\infty,a_{j-1}]$. Then the family of all such minimizing geodesics emanating from all the points on $\varphi^{-1}(\{a_{j+1}\})$ to the points on $\varphi^{-1}(\{a_{j-1}\})$ simply covers the set $\varphi^{-1}[b_{j-1},b_j]$, $j=1,2,\cdots,k$. We define $p_j:=T(p_j',q_j)\cap\varphi^{-1}(\{b_j\})$ and $p_{j-1}:=T(p_j',q_j)\cap\varphi^{-1}(\{b_{j-1})$. Once the point $p_{j-1}$ has been defined, we then choose $p_{j-1}'\in\varphi^{-1}(\{a_j\})$ and $q_{j-1}\in\varphi^{-1}(\{a_{j-2}\})$ in such a way that $T(p_{j-1}', q_{j-1})$ realizes the distance  $d(p_{j-1}',q_{j-1})=d(p_{j-1},\varphi^{-1}(-\infty,a_{j-2}]) $ and it contains $p_{j-1}$ in its interior. We thus obtain the inductive construction of a sequence $\{T(p_j',q_j) \,|\, j=1,\cdots,k\}$ of minimizing geodesics. 


\bigskip


\setlength{\unitlength}{1cm} 
  \begin{figure}[h]
    \begin{center}
\begin{picture}(6, 12)
\thicklines
\qbezier[40](-1,0)(3,1)(7,0)
\qbezier[40](-1,2)(3,3)(7,2)
\qbezier[40](-1,7)(3,8)(7,7)
\qbezier[40](-1,9)(3,10)(7,9)
\qbezier[40](-1,11)(3,12)(7,11)

\thinlines
\qbezier(-1,1)(3,2)(7,1)
\qbezier(-1,3)(3,4)(7,3)


\qbezier(-1,6)(3,7)(7,6)
\qbezier(-1,8)(3,9)(7,8)
\qbezier(-1,10)(3,11)(7,10)

\put(7.2,-0.1){$\varphi^{-1}(b_{-1})$}
\put(7.2,0.9){$\varphi^{-1}(a_{0})=\varphi^{-1}(a)$}
\put(7.2,1.9){$\varphi^{-1}(b_{0})$}
\put(7.2,2.9){$\varphi^{-1}(a_{1})$}
\put(7.2,5.9){$\varphi^{-1}(a_{k-2})$}
\put(7.2,6.9){$\varphi^{-1}(b_{k-2})$}
\put(7.2,7.9){$\varphi^{-1}(a_{k-1})$}
\put(7.2,8.9){$\varphi^{-1}(b_{k-1})$}
\put(7.2,9.9){$\varphi^{-1}(a_{k})=\varphi^{-1}(b)$}
\put(7.2,10.9){$\varphi^{-1}(b_{k})$}

\qbezier(3,0.5)(3,3)(4.5,3.45)
\qbezier(3.5,1.5)(3,4)(3.5,4.8)
\put(3.4,5){$\vdots$}
\qbezier(3.5,5.5)(3,8)(4.5,8.45)
\qbezier(3.8,6.5)(3,9)(4,10.45)
\qbezier(4,8.45)(3,10)(4,11.45)

\put(4,11.7){$p_k$}
\put(4,10.7){$p_{k-1}'$}
\put(2.5,9.7){$p_{k-1}$}
\put(4.5,8.7){$p_{k-2}'$}
\put(2.5,7.7){$p_{k-2}$}
\put(3.6,8.2){$q_{k}$}
\put(2.5,5.5){$p_{k-3}$}
\put(4,6.2){$q_{k-1}$}
\put(3,4.7){$p_1$}
\put(4.3,3.7){$p_0'$}
\put(2.85,2.7){$p_0$}
\put(3.6,1.2){$q_1$}
\put(2.5,1.7){$q_0'$}
\put(2.5,0.7){$q_0$}



\linethickness{0.5mm} 
\qbezier(3.56,9.45)(3.3,10.5)(4,11.45)
\qbezier(3.52,7.5)(3.3,8.5)(3.56,9.45)
\qbezier(3.52,5.5)(3.2,6.5)(3.52,7.5)
\qbezier(3.35,2.45)(3.1,3.8)(3.5,4.8)

\end{picture}
\caption{The broken geodesic $T(p_k)$.}\label{The broken geodesic }
    \end{center}
  \end{figure}
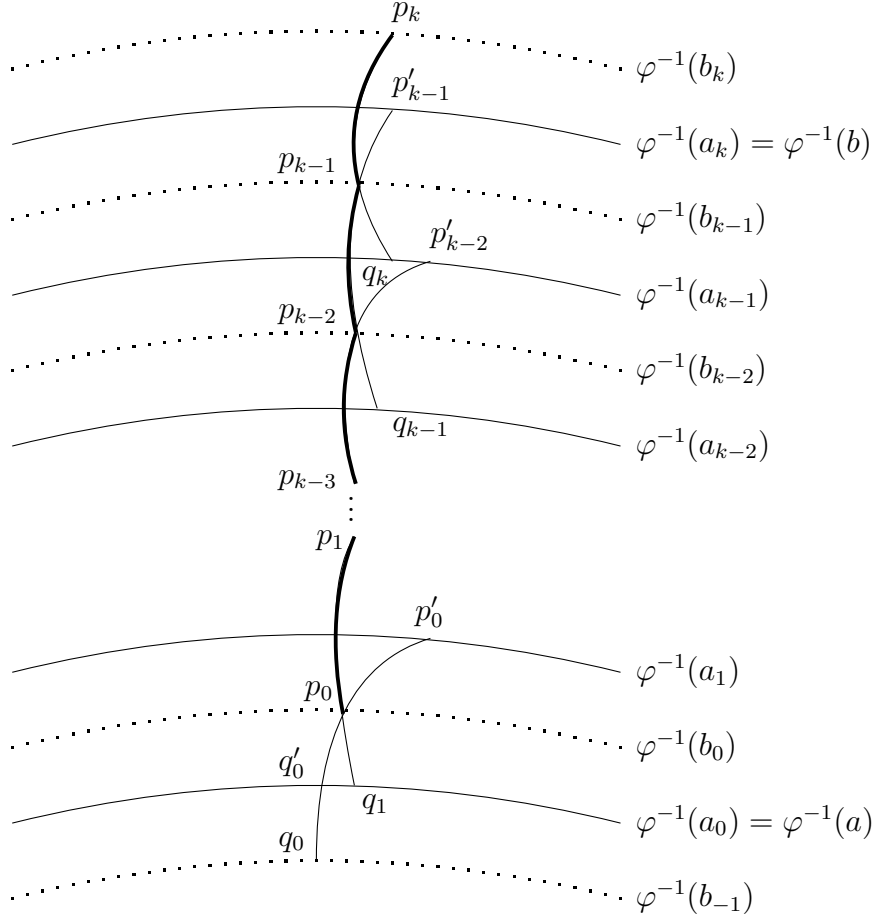

We finally choose a point $p_0'\in\varphi^{-1}(\{a_1\})$ and $q_0\in\varphi^{-1}(\{b_{-1}\})$ such that $T(p_0',q_0)$ is a unique minimizing geodesic with $q_0$ being the foot of $p_0'$ on $\varphi^{-1}(-\infty,b_{-1}]$. If we set $q_0':=\varphi^{-1}(\{a_0\})\cap T(p_0',q_0)$, then $d(p_0,q_1)\le d(p_0,q_0')$ follows from the fact that $q_1$ is the foot of $p_0$ on $\varphi^{-1}(-\infty,a_0]$. Therefore the slope inequality along $T(p_0',q_0)$ implies 
     \[        \frac{a_0-b_0}{d(p_0,q_1)}\le\frac{a_0-b_0}{d(p_0,q_0')}\le\frac{b_{-1}-a_0}{d(q_0',q_0)},   \]
and hence there exists a positive number 
         \[   
  \Delta^b_a(K):=\min\{\frac{a-b_{-1}}{d(q_0',\varphi^{-1}(-\infty,b_{-1}])}\; |\; q_0'\in\varphi^{-1}(\{a\})\}
          \]
with the property that all the slopes of $\varphi\circ T(p_j',q_j)$ for every $j=0,1,\cdots,k$ are negative and bounded above by $-\Delta^b_a(K)$.\par

We define a broken geodesic $T(p_k):=T(p_k,p_{k-1})\cup \cdots\cup T(p_1,p_0)$, $p_k\in\varphi^{-1}(\{b_k\})$, with its break points at $p_j\in\varphi^{-1}(\{b_j\})$, $j=0,1,\cdots,k-1$ in such a way that each $T(p_j,p_{j-1})$ is a proper subarc of a unique minimizing geodesic $T(p_j',q_j)$, where $q_j$ is the foot of $p_j'$ on $\varphi^{-1}(-\infty,a_{j-1}]$ (see Figure \ref{The broken geodesic }). Then $T(p_{j-1},p_{j-2})$ is a proper subarc of $T(p'_{j-1},q_{j-1})$. 
Clearly, the convex function along $T(p_k)$ is monotone strictly decreasing, since
the slopes along $\varphi\circ T(p_k)$ are all bounded above by $-\Delta^b_a(K)$. We then observe from the construction that the family of all the broken geodesics emanating from all points on $\varphi^{-1}(\{b_k\})$ and ending at points on 
$\varphi^{-1}(\{b_{-1}\})$ simply covers $\varphi^{-1}[a,b]$.  The desired homeomorphism $\Phi^b_a$ is now obtained by defining $\Phi^b_a(x,t)$ as the intersection of a $T(x)$ emanating from $x$ :  $\Phi^b_a(x,t)=T(x)\cap \varphi^{-1}(\{t\})$.    
$\qedd$
 \end{proof}
\begin{proposition}\label{Prop:diameter}
{\rm
Assume that all the levels of $\varphi$ are compact. Then the diameter function
$\delta:\varphi(M)\to\R$ defined by
   \[      \delta(a):=\sup\{d(x,y) \,|\, x,y\in\varphi^{-1}(\{a\}),\quad a\in\varphi(M)     \]
is locally Lipschitz.
}
\end{proposition}

\begin{proof}
Let $\inf_M\varphi<a<b<\infty$ and $r=r(M^b_a(\varphi))$ the convexity radius over $M^b_a(\varphi)$. Let $x,y\in\varphi^{-1}(\{s\})$ for $s\in[a,b)$ be such that $d(x,y)=\delta(s)$. Proposition \ref{Prop:homeomorphism} then implies that there are points $x', y'\in M^b_b(\varphi)$ such that $\Phi^b_a(x',s)=x$ and $\Phi^b_a(y',s)=y$. Moreover 
we have $\Phi^b_a(x',t)=T(x')\cap M^t_t(\varphi)$ and $\Phi^b_a(y',t)=T(y')\cap M^t_t(\varphi)$ and the length $L(T(p)|_{[s,t]})$ of $T(p)|_{[s,t]}$, for $p\in M_b^b(\varphi)$ and 
 for every $a\le s<t\le b$, is bounded above by 
     \begin{equation}\label{eq:length}
          L(T(p)|_{[s,t]})\le |t-s|/\Delta^b_a(M^b_a(\varphi)).        
      \end{equation}    
We therefore have by setting $\lambda=\lambda(M^b_a(\varphi))$ the reversibility constant on $M^b_a(\varphi)$, 
 \begin{align*}
     \delta(s) &=d(x,y)\le d(x,\Phi^b_a(x',t))+d(\Phi^b_a(x',t),\varphi^b_a(y',t))+d(\Phi^b_a(y',t),y)  \\
     &\le \lambda|t-s|/\Delta^b_a(M^b_a(\varphi))+\delta(t)+|t-s|/\Delta^b_a(M^b_a(\varphi))  \\
     &=(1+\lambda)|t-s|/\Delta^b_a(M^b_a(\varphi))+\delta(t).
   \end{align*}
Similarly, we obtain by choosing $x,y\in\varphi^{-1}(\{t\})$, $d(x,y)=\delta(t)$,
    \begin{align*}
       \delta(t) &=d(x,y)\le d(x,\Phi^b_a(x',s))+d(\Phi^b_a(x',s),\Phi^b_a(y',s))+d(\Phi^b_a(y',s),y)  \\
     &\le (1+\lambda)|t-s|/\Delta^b_a(M^b_a(\varphi))+\delta(s),
     \end{align*}
and hence,
   \[     |\delta(t)-\delta(s)|\le(1+\lambda)|t-s|/\Delta^b_a(M^b_a(\varphi)).    \]
$\qedd$
\end{proof}
\medskip\noindent
{\it The proof of Theorem~\ref{Th:homeomorphism}}.\par
We first assume that  $\inf_M\varphi$ is not attained. Let $\{a_j\}_{j=0,\pm 1,\cdots}$ be a monotone increasing sequence of real numbers with $\lim_{j\to -\infty}a_j=\inf_M\varphi$ and $\lim_{j\to\infty}a_j=\infty$.  We then apply Proposition \ref{Prop:homeomorphism} to each integer $j$ and obtain a homeomorphism $\Phi^{j+1}_j:\varphi^{-1}(\{a_{j+1}\})\times(a_j,a_{j+1}] \to M_{a_j}^{a_{j+1}}$ such that 
 \[ \varphi\circ\Phi^{j+1}_j(x,t)=t,\quad x\in\varphi^{-1}(\{a_{j+1}\}),\quad t\in(a_j,a_(j+1] \]

The composition of these homeomorphisms gives the desired homeomorphisms
$\varphi:\varphi^{-1}(\{a\})\times(\inf_M\varphi,\infty)\to M$.\par
If $\lambda:=\inf_M\varphi$ is attained, then $M^\lambda_\lambda(\varphi)$ is a $k$-dimensional totally geodesic submanifold which is totally convex and $0\le k\le \dim M-1$. A tubular neighborhood $B(M^\lambda_\lambda(\varphi),r(M^\lambda_\lambda(\varphi)))$ around the minimum set is a normal bundle over $M^\lambda_\lambda(\varphi)$ in $M$ and its boundary $\partial B(M^\lambda_\lambda(\varphi),r(M^\lambda_\lambda(\varphi)))$ is homeomorphic to a level of $\varphi$. Therefore $M$ is homeomorphic to the normal bundle over the minimum set in $M$. This proves Theorem 
~\ref{Th:homeomorphism}.\par\medskip
$\qedd$
\begin{remark}
{\rm
Under the assumption in Theorem \ref{Th:homeomorphism}, it is not certain
whether or not $\lim_{t\to\inf_M\varphi}\delta(t)=\infty$. It might happen that every level
set above infimum is compact but the minimum set is non-compact. We do not know such an example on a Finsler manifold. \par
}
\end{remark}
\begin{remark}\label{Rem:important}
{\rm
The basic difference of treatments of convex functions between Riemannian and Finsler geometry can be interpreted as follows.\par
In the case where $\varphi:(M,g)\to\R$ is a convex function with non-compact levels, 
the homeomorphism $\Phi^b_a:M^b_b(\varphi)\times[a,b]\to M^b_a(\varphi)$ is obtained as follows. Fix a point $p\in M^a_a(\varphi)$ and a sequence of $R_j$-balls centered at $p$ : $\{B(p,R_j) \,|\, j=1,2,\cdots\}$ with $\lim_{j\to\infty}R_j=\infty$. Setting 
$K_j$ for $j=1,2\cdots$ the closure of $B(p,R_j)$, we find a sequence of constants $\Delta_j:=\Delta^b_a(K_j)$. If $x\in K_j\cap M^b_b(\varphi)$ is a fixed point, we then have a broken geodesic $T(x):=T(x_k,x_{k-1})\cup\cdots\cup T(x_1,x_0)$ as obtained in the proof of Proposition \ref{Prop:homeomorphism}, where $x_0\in M^a_a(\varphi)$. The special properties of Riemannian distance function now applies to $T(x_j,x_{j-1}):[0,d(x_j,x_{j-1})]\to (M,g)$ to obtain that the distance function from $p\in M_a^a(\varphi)$, 
$t\mapsto d(p,T(x_j,x_{j-1}))(t)$ is strictly monotone decreasing. Here $T(x_j,x_{j-1})$ is parameterized by arc-length such that $T(x_j,x_{j-1})(0)=x_j$ and $T(x_j,x_{j-1})(d(x_j,x_{j-1}))=x_{j-1}$. Therefore we observe that $T(x)$ is contained entirely in $K_j$ and moreover the length 
 $L(T(x))$ of $T(x)$ satisfies
      \[       L(T(x))\le(b-a)/\Delta_j,\quad \forall x\in K_j\cap M^b_a(\varphi).     \] 
If $y_0\in M^a_a(\varphi)\cap K_j$ is an arbitrary fixed point, Proposition \ref{Prop:homeomorphism} again implies that there exists a point $y=y_m\in M^b_b(\varphi)$ such that $T(y)=T(y_k,y_{k-1})\cup \cdots \cup T(y_1,y_0)$ has length at most $(b-a)/\Delta_j$. Therefore we have 
    \[      d(p,y)< R_j+(b-a)/\Delta_j+1.      \]
We therefore observe that the correspondence between $M^b_b(\varphi)$ and $M^a_a(\varphi)$, $x\mapsto x_0$ through $T(x)$ is bijective, and the desired homeomorphism is constructed.\par
However in the Finslerian  case where all the levels of a convex function $\varphi:(M,F)\to\R$ are non-compact, the correspondence between $M^b_b(\varphi)$ and $M^a_a(\varphi)$, $x\mapsto x_0$ through $T(x)=T(x_k,x_{k-1})\cup\cdots \cup T(x_1,x_0)$ may not be obtained. In fact, the monotone decreasing property of $t\mapsto d(p,T(x_j,x_{j-1})$ might not hold for a Finsler metric. Therefore $T(x)$ for a point $x\in K_j\cap M^b_b(\varphi)$ may not necessarily be contained in $K_j$, and hence, we may fail in controlling the length of $T(x)$ in terms of $\Delta_j$. By the same reason, we cannot prove the monotone non-decreasing property of the diameter function for compact levels of a convex function $\varphi:(M,F)\to\R$.       
}
\end{remark}

\bigskip
\section{Proof of Theorem \ref{Th:disconnected}}
We take a minimizing geodesic $\sigma:[0,\ell]\to M$ such that $\sigma(0)$ and $\sigma(\ell)$ belong to  distinct components of $M^c_c(\varphi)$. \par\medskip
For the proof of (1), we assert that $\inf_M\varphi=\inf_{0\le t\le\ell}\varphi\circ\sigma(t)$. Suppose $b:=\inf_{0\le t\le\ell}\varphi\circ\sigma(t)>\inf_M\varphi$. Since $\varphi$ is locally non-constant, we may assume without loss of generality that $b:=\inf_{0\le t\le\ell}\varphi\circ\sigma(t)$ is attained at a unique point, say, $q=\sigma(\ell_0)$.

 Setting $r=r(\sigma(\ell_0))$, we find a number $a\in(\inf_M\varphi,b)$ such that there is a unique foot $p\in M^a_a(\varphi)$ of $q$ on $M^a_a(\varphi)$, namely $d(\sigma(\ell_0),M^a_a(\varphi))=d(\sigma(\ell_0),p)$.

 Let $\alpha:[0,d(q,p)]\to M$ be a unique minimizing geodesic with $\alpha(0)=q$, $\alpha(d(q,p))=p$. The points on $\alpha(t),\,0\le t\le d(q,p)$ can be joined to $q_\pm:=\sigma(\ell_0\pm r)$ by a unique minimizing geodesic $\g_{\alpha(t)q_\pm}:[0,d(\alpha(t),q_\pm)]\to B(q;r)$ with $\gamma_{\alpha(t)q_\pm}(0)=\alpha(t)$, 
 $\gamma_{\alpha(t)q_\pm}(d(\alpha(t)),q_\pm)=q_\pm$.

 Since $\varphi(q_\pm)>b$, the right hand derivative of $\varphi\circ\g_{\alpha(t)q_\pm}$ at $d(\alpha(t),q_\pm)$ is bounded below by 
    \[        (\varphi\circ\g_{\alpha(t)q_\pm})'_+(\varphi\circ
\gamma_{\alpha(t)q_\pm}(d(\alpha(t)),q_\pm))
>\frac{\varphi(q_{\pm})-b}{2r}>0.      \]
Thus, 
for every $t\in[0,d(q,p)]$,
 $\g_{\alpha(t)q_\pm}$
 meets $M^c_c(\varphi)$ at  
$\g_{\alpha(t)q_\pm}(u^\pm(t))$ with
     \[       u^\pm(t)\le \frac{2r(c-a)}{\varphi(q_\pm)-b}+2r,       \]
and hence there are curves $C^\pm_0:[0,d(q,p)]\to M^c_c(\varphi)$ with $C^+_0(0)=\sigma(\ell)$, $C^-_0(0)=\sigma(0)$ and $C^+_0(d(q,p))=\g_{pq_+}(u^+(d(q,p)))$, $C^-_0(d(q,p))=\g_{pq_-}(u^-(d(q,p)))$. Let $\tau_t:[0,d(p,\sigma(t))]\to M$ for $t\in[\ell_0-r,\ell_0+r]$ be a minimizing geodesic with $\tau_t(0)=p$, $\tau_t(d(p,\sigma(t))=\sigma(t)$.  Every $\tau_t$ meets $M^c_c(\varphi)$ at a parameter value $\le (2rc)/(b-a)$, and hence we have a curve $C_1:[\ell_0-r,\ell_0+r]\to M^c_c(\varphi)$ such that 
     \[       C_1(t)=\tau_t[0,\frac{2rc}{b-a}]\cap M^c_c(\varphi).      \]
Thus, considering the union  $C_0^-\cup C_1\cup (C_0^+)^{-1}$, it follows that $\sigma(0)$ can be joined to $\sigma(\ell)$ in $M^c_c(\varphi)$, a contradiction. This proves (1) (see Figure \ref{The proof of Theorem 1.1.}).

\bigskip


\setlength{\unitlength}{1.5cm} 
 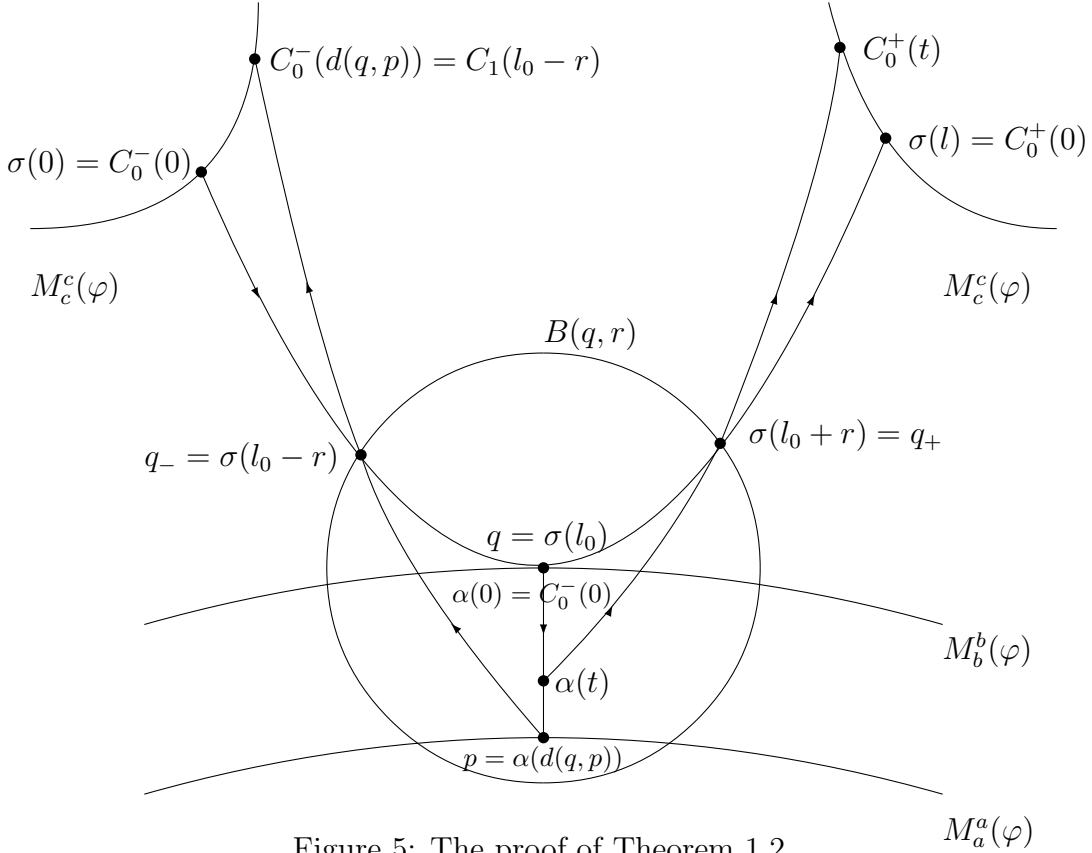
\begin{figure}[h]
    \begin{center}
\begin{picture}(6, 7)
\put(3,2){\circle*{0.1}}
\put(2.5,2.2){$q=\sigma(l_0)$}
\put(3,0.5){\circle*{0.1}}
\put(2.3,0.25){\footnotesize ${p=\alpha(d(q,p))}$}
\put(3,4){$B(q,r)$}

\qbezier(3,0.5)(1.7,2)(1.4,3)

\qbezier(1.4,3)(1,4)(0.47,6.5)
\put(0.47,6.5){\circle*{0.1}}
\put(0.6,6.4){$C_{0}^{-}(d(q,p))=C_{1}(l_{0}-r)$}

\qbezier(4.9,2.0)(4.9,2.787)(4.3435,3.3435) 
\qbezier(4.3435,3.3435)(3.787,3.9)(3.0,3.9) 
\qbezier(3.0,3.9)(2.213,3.9)(1.6565,3.3435) 
\qbezier(1.6565,3.3435)(1.1,2.787)(1.1,2.0) 
\qbezier(1.1,2.0)(1.1,1.213)(1.6565,0.6565) 
\qbezier(1.6565,0.6565)(2.213,0.1)(3.0,0.1) 
\qbezier(3.0,0.1)(3.787,0.1)(4.3435,0.6565) 
\qbezier(4.3435,0.6565)(4.9,1.213)(4.9,2.0)

\qbezier(-0.5,1.5)(3,2.5)(6.5,1.5)
\put(6.5,1.2){$M_b^b(\varphi)$}

\qbezier(-0.5,0)(3,1)(6.5,0)
\put(6.5,-0.4){$M_a^a(\varphi)$}
\put(6.5,4.4){$M_c^c(\varphi)$}
\put(-1.5,4.4){$M_c^c(\varphi)$}
\put(6.2,5.7){$\sigma(l)=C_{0}^{+}(0)$}
\put(6,5.8){\circle*{0.1}}
\put(5.6,6.6){\circle*{0.1}}
\put(5.8,6.5){$C_0^{+}(t)$}
\put(-1.7,5.5){$\sigma(0)=C_{0}^{-}(0)$}

\put(0.5,4.4){\vector(1,-2){0.01}}
\put(5.37,4.4){\vector(1,2){0.01}}
\put(5.04,4.4){\vector(1,3){0.01}}
\put(3.6,1.65){\vector(1,2){0.01}}

\qbezier(-1.5,5)(0.5,5)(0.5,7)
\qbezier(5.5,7)(6,5)(7.5,5)
\qbezier(0,5.5)(3,-1.6)(6,5.8)
\put(0,5.5){\circle*{0.1}}

\qbezier(4.55,3.1)(5.5,5.5)(5.6,6.6)
\put(4.55,3.1){\circle*{0.1}}
\put(4.8,3.1){$\sigma(l_0+r)=q_+$}
\put(1.4,3){\circle*{0.1}}
\put(-0.5,2.9){$q_-=\sigma(l_0-r)$}
\put(0.93,4.5){\vector(-1,3){0.01}}

\put(2.195,1.5){\vector(-1,1){0.01}}

\qbezier(3,1)(4,2)(4.55,3.1)

\qbezier(3,0.5)(3,1)(3,2)
\put(3,1.5){\vector(0,-1){0.1}}
\put(3,1){\circle*{0.1}}
\put(3.1,0.9){$\alpha(t)$}

\put(2.2,1.7){\footnotesize $\alpha(0)=C_{0}^{-}(0)$}

\end{picture}
\caption{The proof of Theorem \ref{Th:disconnected}.}\label{The proof of Theorem 1.1.}
    \end{center}
  \end{figure}



We next prove (2). Let $\lambda:=\inf_M\varphi$. Clearly $M^\lambda_\lambda(\varphi)$ is totally convex, and hence Proposition 2.3 implies that $M^\lambda_\lambda(\varphi)$ carries the structure of a smooth totally geodesic submanifold. 

Suppose that $\dim M^\lambda_\lambda(\varphi)<n-1$, then the normal bundle is connected, and at each point $p\in M^\lambda_\lambda(\varphi)$ the indicatrix $\Sigma_p\subset T_pM$ has the property that $\Sigma_p\setminus\Sigma_p(M^\lambda_\lambda\varphi))$ is arcwise connected. Here, $\Sigma_p(M^\lambda_\lambda(\varphi))\subset\Sigma_p$ is the indicatrix at $p$ of $M^\lambda_\lambda(\varphi)$. Choose points $q_0$ and $q_1$ on distinct components of $M^c_c(\varphi)$, and an interior point $p\in M^\lambda_\lambda(\varphi)$. If $\g_i:[0,d(p,q_i)]\to M$ for $i=0,1$ is a minimizing geodesic with $\g_i(0)=p$, $\g_i(d(p,q_i))=q_i$, then $\dot\g_0(0)$ and $\dot\g_1(0)$ are joined by a curve $\Gamma:[0,1]\to\Sigma_p\setminus\Sigma_p(M^\lambda_\lambda(\varphi))$ such that $\Gamma(0)=\dot\g_0(0)$, $\Gamma(1)=\dot\g_1(0)$. The same method as developed in the proof of (1) yields a continuous $1$-parameter family of geodesics $\g_t:[0,\ell_t]\to M$ with $\g_t(0)=p$, $\dot\g_t(0)=\Gamma(t)$ and $\g_t(\ell_t)\in M^c_c(\varphi)$ for all $t\in[0,1]$. Thus we have a curve $t\mapsto \gamma_t(\ell_t)$
in $M^c_c(\varphi)$ joining $q_0$ to $q_1$, a contradiction. This proves $\dim M^\lambda_\lambda(\varphi)=n-1$. \par
We use the same idea for the proof of $M^\lambda_\lambda(\varphi)$ having no boundary. In fact, the tangent cone of $M^\lambda_\lambda(\varphi)$ at a boundary point $x$ (supposing that the boundary is non-empty) is contained entirely in a closed half space of $T_xM^\lambda_\lambda(\varphi)$, and hence $\Sigma_x\setminus\Sigma_x(M^\lambda_\lambda(\varphi))$ is arcwise connected. A contradiction is derived by constructing a curve in $M^c_c(\varphi)$ joining $q_0$ to $q_1$. This proves (2).\par
The triviality of the normal bundle over $M^\lambda_\lambda(\varphi)$ in $M$ is now clear.\par
We finally prove (4).  Suppose that $M^a_a(\varphi)$ for some $a\in\varphi(M)$ has at least three components. Let $q_1, q_2, q_3\in M^a_a(\varphi)$ be taken from distinct components, and $p\in M^\lambda_\lambda(\varphi)$. Let $\g_i:[0,d(p,q_i)]\to M$ for $i=1,2,3$ be minimizing geodesics with $\g_i(0)=p$, $\g_i(d(p,q_i))=q_i$. As is shown in (3), 
since the normal bundle over $M^\lambda_\lambda(\varphi)$ in $M$ is trivial, it follows that  
 $\Sigma_p\setminus\Sigma_p(M^\lambda_\lambda(\varphi))$ has exactly tow components. Two of the three initial vectors, say $\dot\g_1(0)$ and $\dot\g_2(0)$ belong to the same component of  $\Sigma_p\setminus\Sigma_p(M^\lambda_\lambda(\varphi))$. Then the same technique as developed in the proof of $\dim M^\lambda_\lambda(\varphi)=n-1$ applies, and $q_1$ is joined to $q_2$ by a curve in $M^a_a(\varphi)$, that is a contradiction. This proves (4).
$\qedd$
\endproof



       \section{Ends of $(M,F)$}

An {\it end} $\e$ of a noncompact manifold $X$ is an assignment to each compact set $K\subset X$ a component $\e(K)$ of $X\setminus K$ such that $\e(K_1)\supset\e(K_2)$ if $K_1\subset K_2$. Every non-compact manifold has at least one end. For instance, $\R^{n}$ has one end if $n>1$ and two ends if $n=1$. 

In the present section we discuss the number of ends of $(M,F)$ admitting a convex function, namely we will prove Theorem~\ref{Th:ends}.  As is seen in the previous section, it may happen that a convex function $\varphi:(M,F)\to\R$ has both compact and non-compact levels simultaneously. In this section let $\{K_j\}_{j=1,2,\cdots}$ be an increasing sequence of compact sets such that $\lim_{j\to \infty}K_j=M$.\par\medskip

\noindent
We prove Theorem~\ref{Th:ends}-(A1).\par
Theorem \ref{Th:disconnected} (1) then implies that $\varphi$ attains its infimum $\lambda:=\inf_M\varphi$. 
For an arbitrary given compact set $A\subset M$, there exists a number $a\in\varphi(M)$ such that $M^a_a(\varphi)$ has two components and $A\subset \varphi^{-1}[\lambda, a]$. Then $M\setminus A$ contains two unbounded open sets $\varphi^{-1}(a,\infty)$, proving (A-1).\par
\medskip\noindent


We prove Theorem~\ref{Th:ends}-(A2). 
Suppose that $M$ has more than one end. There is a compact set $K\subset M$ such that $M\setminus K$ has at least two unbounded components, say, $U$ and $V$. Setting $a:=\min_K\varphi$ and $b:=\max_K\varphi$, we have 
   \[     \lambda\le a<b<\infty.     \]
We assert that 
     \[    M^\lambda_\lambda(\varphi)\cap U\neq\emptyset,\quad M^\lambda_\lambda(\varphi)\cap V\neq\emptyset,\quad M^\lambda_\lambda(\varphi)\cap K\neq\emptyset.
      \]
In order to prove that $M^\lambda_\lambda(\varphi)\cap K\neq\emptyset$, we
suppose that $\lambda<a$. 

Suppose the contrary, namely $M^\lambda_\lambda(\varphi)\cap K=\emptyset$. Once 
$M^\lambda_\lambda(\varphi)\cap K\neq\emptyset$ has been established, 
 it will turn out that $M^\lambda_\lambda(\varphi)$ intersects all the unbounded components of $M\setminus K$. 

Without loss of generality we may suppose that $M^\lambda_\lambda(\varphi)\subset U$. From Theorem \ref{Th:disconnected} (3) it follows that 
 $M\setminus M^\lambda_\lambda(\varphi)=M_-\cup M_+$ (disjoint union with 
$\partial M_+=\partial M_-=M^\lambda_\lambda(\varphi)$). 

Setting $M_-\subset U$, we observe that $K\cup V\subset M_+$.

If $b_1>b$, then $M_-$ contains a component of $M_{b_1}^{b_1}(\varphi)$ and another component of $M_{b_1}^{b_1}(\varphi)$ is contained entirely in $V$. We then observe that if $\sup_{U\setminus M_-}\varphi=\infty$, then $U\setminus M_-$ contains a component of $M_{b_1}^{b_1}(\varphi)$, for $\varphi$ takes value $\leq b$ on $\partial(U\setminus M_-)$ and $M_{b_1}^{b_1}(\varphi)$ 
does not meet the boundary of $U\setminus M_- $. 
This contradicts to the Theorem \ref{Th:disconnected} (4), for  $\partial M_{b_1}^{b_1}(\varphi)$ has at least three components. 
Therefore we have $\sup_{U\setminus M_-}\varphi<\infty$.


\bigskip


\setlength{\unitlength}{1cm} 
  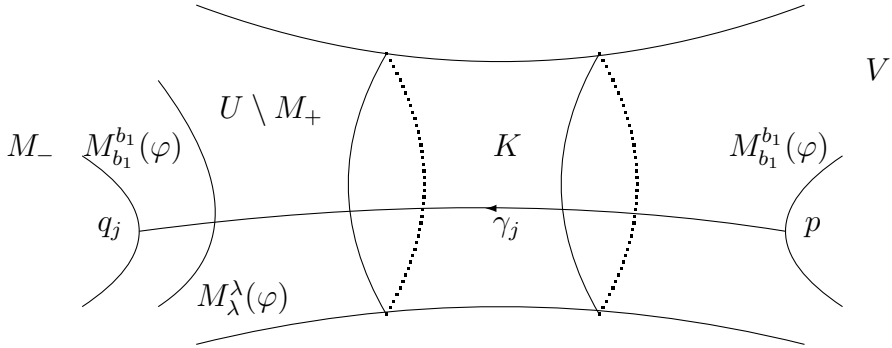
\begin{figure}[h]
    \begin{center}
\begin{picture}(6, 4)

\qbezier(-1,0)(3,1)(7,0)
\qbezier(-1,4.5)(3,3)(7,4.5)

\qbezier(1.5,0.4)(0.5,2.125)(1.5,3.85)
 \thicklines
\qbezier[40](1.5,0.4)(2.5,2.125)(1.5,3.85)
\thinlines
\put(2.9,2.5){$K$}
\put(2.9,1.5){$\gamma_j$}
\put(2.9,1.8){\vector(-1,0){0.1}}

\qbezier(4.3,0.4)(3.3,2.125)(4.3,3.85)
\thicklines
\qbezier[40](4.3,0.4)(5.3,2.125)(4.3,3.85)
\thinlines

\put(7.8,3.5){$V$}
\put(-0.7,3){$U\setminus M_+ $}
\put(-3.5,2.5){$M_{-}$}

\qbezier(-2.5,2.5)(-1,1.5)(-2.5,0.5)
\qbezier(-1.5,3.5)(0,1.5)(-1.5,0.5)
\qbezier(7.5,2.5)(6,1.5)(7.5,0.5)

\put(-1,0.5){$M_\lambda^\lambda(\varphi)$}
\put(-2.5,2.5){$M_{b_1}^{b_1}(\varphi)$}
\put(6,2.5){$M_{b_1}^{b_1}(\varphi)$}
\put(7,1.5){$p$}
\put(-2.3,1.5){$q_j$}

\qbezier(-1.75,1.5)(2.9,2.125)(6.75,1.5)
\end{picture}
\caption{The proof of Theorem~\ref{Th:ends}-(A2).}\label{The proof of Theorem 1.2-(A2)}
    \end{center}
  \end{figure}




Let $\{q_j\}\subset M_\lambda^\lambda(\varphi)$ be a divergent sequence of points and let us fix a point $p\in M_{b_1}^{b_1}(\varphi)\subset V$. Let $\gamma_j:[0,d(p,q_j)]\to M\setminus M_-$ be a minimizing geodesic with 
$\gamma_j(0)=p$, $\gamma_j(d(p,q_j))=q_j$, $j=1,2,\dots$. Clearly $\gamma_j$ passes through a point on $K$ and $\varphi\circ\gamma_j$ is bounded above by $b_1$. If $\gamma:[0,\infty)\to M\setminus M_-$ is a ray with $\dot{\gamma}(0)=\lim_{j\to \infty}\dot{\gamma}_j(0)$, then $\varphi\circ\gamma$ is constant on $[0,\infty)$ and $\varphi\circ\gamma(t)=b_1$ for all $t>b_1$. This is a contradiction to the choice of $b=\max_K\varphi$, for $\gamma$ passes through a point on $K$ at which $\varphi$ takes the value $b_1$. 
This proves the assertion (see Figure \ref{The proof of Theorem 1.2-(A2)}).\par

We next assert that if $b_1>b$ is fixed, then $M^{b_1}_{b_1}(\varphi)$ has at least four components. In fact, we observe from $M^{b_1}_{b_1}(\varphi)\cap K=\emptyset$, each unbounded component of $(M\setminus M^\lambda_\lambda(\varphi))\cap(M\setminus K)$ contains a component of $M^{b_1}_{b_1}(\varphi)$. This contradicts Theorem \ref{Th:disconnected}  (4), and (A-2) is  proved.
\par\medskip\noindent

The proof of (A3) is a consequence of (D2), and given after the proof of (D2).

We prove Theorem ~\ref{Th:ends}-(B1). 
From assumption of (B1), it follows that  $\varphi^{-1}[\inf_M\varphi,b_j]$ is compact for all $j$, where $\{b_j\}$ is a monotone divergent sequence. Then $K_j:=\varphi^{-1}[\lambda,b_j]$ is monotone increasing and $\lim_{j\to\infty}K_j=M$. Clearly $M\setminus K_j$, for every $j=1,2,\dots$,  
contains a unique unbounded domain $\varphi^{-1}(b_j,\infty)$. This proves Theorem ~\ref{Th:ends}-(B1).
\par\medskip\noindent
We prove Theorem ~\ref{Th:ends}-(B2).
From assumption of (B2), we have monotone sequences $\{a_j\}$ and $\{b_j\}$ such that
      \[    
\lim_{j\to\infty}a_j=\inf_M\varphi,\quad\lim_{j\to\infty}b_j=\infty,\quad [a_j,b_j]=\varphi(K_j),\quad j=1,2,\cdots.    
      \]
Then $M\setminus K_j$ for all large number $j$ contains two unbounded domains
   \[      M\setminus K_j\supset\varphi^{-1}(b_j,\infty)\cup\varphi^{-1}(\inf_M\varphi,a_j).    \]
This proves that $M$ has exactly two ends.
\par\medskip
We prove Theorem ~\ref{Th:ends}-(C).
We first prove (C) under an additional assumption that $\lambda:=\inf_M\varphi$ is attained. Suppose that $M$ has more than one end. Using the same notation as in the proof of (A2), 
    \[       \lambda:=\inf_M\varphi,\quad a:=\min_K\varphi,\quad b:=\max_K\varphi,    \]
where $K\subset M$ is a compact set such that $M\setminus K$ has at least two unbounded components $U$ and $V$. \par
We first assert that $K\cap M^\lambda_\lambda(\varphi)\neq\emptyset$. In fact, supposing that $K\cap M^\lambda_\lambda(\varphi)=\emptyset$ we find a component $V$ of $M\setminus K$ such that if $b'>b$ then $M_{b'}^{b'}(\varphi)\subset V$ and $M_\lambda^\lambda\subset U$. Here the assumption that all the levels of $\varphi$ are connected is essential. As is seen in the proof of (A2), there exist at least two components of $M^{b'}_{b'}$ for $b'>b$ such that one component of it lies in $U$ and another in $V$. This contradicts to the assumption in (C), and the first assertion is done.

The same proof technique as developed in (A2) implies that
 $M^\lambda_\lambda(\varphi)$ passes through points on $K$, $U$ and $V$. Fix a point $p\in V\cap M^\lambda_\lambda(\varphi)$ and a divergent sequence $\{q_j\}$ of points in $U\setminus M^\lambda_\lambda(\varphi)$ and $\g_j:[0,d(p,q_j)]\to M$ a minimizing geodesic with $\g_j(0)=p$, $\g_j(d(p,q_j))=q_j$. From construction of $\g_j$, we observe that $\varphi\circ\g_j$ is strictly increasing, and hence we find a number $t_j>0$ such that  $\g_j(t_j)\in M^{b'}_{b'}(\varphi)\cap U$.  
 
 From the construction of $\gamma_j$, we observe that $\varphi\circ\gamma$ is strictly increasing, and hence we find a number $t_j>0$ such that $\gamma_j(t_j)\in M_{b'}^{b'}(\varphi)\cap U$. 
 
More precisely, $\gamma_j[0,d(p,q_j)]$ meets $M_\lambda^\lambda(\varphi)$ only at the origin, for $M_\lambda^\lambda(\varphi)$ is totally convex and hence if $\gamma(t_0)\in M_\lambda^\lambda(\varphi)$, for some $t_0\in [0,d(p,q_j))$, then  $\gamma_j[0,d(p,q_j)]$ is contained entirely in $M_\lambda^\lambda(\varphi)$. 
 
 Therefore $M^{b'}_{b'}(\varphi)$ has more than one components (one component in $U$ and another component in $V$), a  contradiction to the assumption in (C). This concludes the proof of (C) in this case.\par\smallskip
We next prove (C) in the case where $\inf_M\varphi$ is not attained. Assume again that $M$ has more than one end. We then have
    \[     \inf_M\varphi<a<b<\infty,\quad a:=\min_K\varphi,\quad b:=\max_K\varphi.      \]
Since all the levels are connected, we find $\inf_M\varphi<a'<a$ and $b<b'$ such that $M^{a'}_{a'}(\varphi)\subset U$ and $M^{b'}_{b'}(\varphi)\subset V$. Let $\{y_j\}\subset M^{b'}_{b'}(\varphi)$ be a divergent sequence of points and fix a point $x\in M^{a'}_{a'}(\varphi)$. Let $\g_j:[0,d(x,y_j)]\to M$ for $j=1,2,\cdots$ be a minimizing geodesic with $\g_j(0)=x$ and $\g_j(d(x,y_j))=y_j$. There exists a ray
$\g:[0,\infty)\to M$ emanating from $x$ such that $\dot\g(0)=\lim_{j\to\infty}\dot\g_j(0)$. Clearly, every $\g_j$ passes through a point on $K$ and hence, so does $\g$. From construction, $\varphi\circ\g:[0,\infty)\to\R$ is bounded from above by $b'$, and hence it is constant. However it is impossible, for $\varphi(x)=a'$ and $\varphi\circ\g(t_0)\ge a>a'$ at a point $\g(t_0)\in K$. This proves (C) in this case.\par
\medskip\noindent

We prove Theorem ~\ref{Th:ends}-(D).

For the proof of (D1), we suppose that $M$ has more than two ends.



\setlength{\unitlength}{1cm} 
  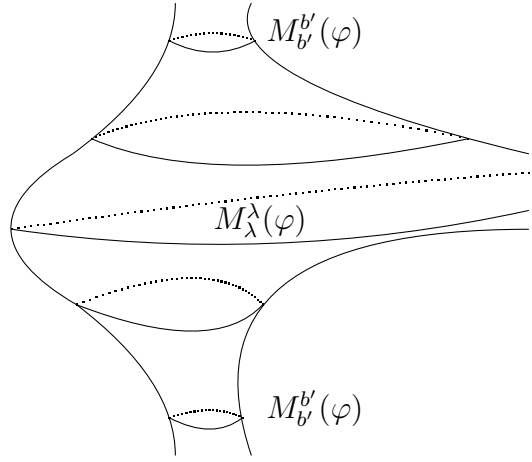
\begin{figure}[h]
    \begin{center}
\begin{picture}(6, 5)

\qbezier(0.65,4)(2,5)(2,6)
\qbezier(0.65,4)(-1,3)(0.7,2)
\qbezier(0.7,2)(2,1)(2,0)

\qbezier(6.65,4)(2.5,5)(3,6)
\qbezier(3,0)(2,3)(6.65,3)
\qbezier(1.92,5.5)(2.5,5.2)(3.05,5.5)
\qbezier[20](1.92,5.5)(2.5,5.7)(3.05,5.5)

\qbezier(0.9,4.2)(2.5,3.5)(5.85,4.2)
\qbezier[60](0.9,4.2)(2.5,4.9)(5.85,4.2)

\qbezier(1.92,0.5)(2.5,0.2)(2.88,0.5)
\qbezier[20](1.92,0.5)(2.5,0.7)(2.88,0.5)

\qbezier(0.7,2)(2.5,1.3)(3.15,2)
\qbezier[40](0.7,2)(2.5,2.7)(3.15,2)

\qbezier(-0.15,3)(3.25,2.5)(6.65,3.25)
\qbezier[80](-0.15,3)(3.25,3.5)(6.65,3.75)

\put(3.2,5.5){$M_{b'}^{b'}(\varphi)$}
\put(3.2,0.5){$M_{b'}^{b'}(\varphi)$}
\put(2.5,3){$M_{\lambda}^{\lambda}(\varphi)$}
\end{picture}
\caption{The proof of Theorem ~\ref{Th:ends}-(D1).}\label{The proof of Theorem 1.2-(D1)}
    \end{center}
  \end{figure}

Let $K\subset M$ be a connected compact subset such that $M\setminus K$ contains at least three unbounded components, say $U$, $V$ and $W$. We may consider that $U$ contains 
$\varphi^{-1}[b',\infty)$, for all $b'>b$. Since all the levels of $\varphi$ are connected, we have
\begin{equation*}
\sup_{M\setminus U}\varphi\leq b.
\end{equation*}

In fact, suppose that there exists a point $x\in M\setminus U$ such that $\varphi(x)=b'$, for some $b'>b$. Then $M_{b'}^{b'}(\varphi)\cap K=\emptyset$ and hence $M_{b'}^{b'}(\varphi)$ 
 is disconnected, a contradiction to the assumption of (D). 
 
 Let $\{x_j\}\subset V$ and $\{y_j\}\subset W$ be two divergent sequences of points and $\gamma_j:[0,d(x_j,y_j)]\to M\setminus U$ a minimizing geodesic joining $x_j$ to $y_j$. Since $\gamma_j$ passes through a point on $K$, there exists a straight line $\gamma:\R\to M\setminus U$ such that $\dot \gamma(0)$ is obtained as the limit of a converging sequence of vectors $\dot{\gamma}_j(t_j)\in K$, for $j=1,2,\dots$. Clearly, $\varphi\circ\gamma:\R\to \R$ is bounded above, and hence constant taking a value 
 $\mu=\varphi\circ\gamma(0)\in [a,b]$. We therefore observe that 
 \begin{equation*}
 M_\mu^\mu(\varphi)\cap K\neq \emptyset,\quad M_\mu^\mu(\varphi)\cap W\neq \emptyset\quad {\rm and } \quad 
 M_\mu^\mu(\varphi)\cap V\neq \emptyset.
 \end{equation*}
 
 We next choose a value $a'\in (\inf_M\varphi,a)$. We may assume without loss of generality that $M_{a'}^{a'}(\varphi)\subset V$. Let $\{z_j\}\subset M_\mu^\mu(\varphi)\cap W$ be a divergent sequence of points and $x\in M_{a'}^{a'}(\varphi)$ an arbitrary fixed point. Let $\sigma_j:[0,d(x,z_j)]\to M\setminus U$ be a minimizing geodesic with $\sigma_j(0)=x$, $\sigma(d(x,z_j))=z_j$, for all $j=1,2,\dots$. Clearly, $\varphi\circ\sigma_j$ is monotone increasing in $W$. Let $\sigma:[0,\infty)\to M$ be a ray such that $\dot\sigma(0)=\lim_{j\to \infty}\dot\sigma_j(0)$. We then observe that $\varphi\circ\sigma$
is monotone increasing on an unbounded interval $[\bar{b},\infty)$ for some $\bar{b}>0$, and
bounded above by $\mu$, and hence it is constant equal to $a'$. Recall that $\varphi\circ \sigma(0)=\varphi(x)=a'$. However this is impossible since $ a'<\min_K\varphi=a$ and $\sigma[0,\infty)$ passes through a point on $K$. We therefore observe that $M\setminus (K\cup U)$ has exactly one end. This proves (D1). 

The proof of (D2) is now clear and omitted here.




\setlength{\unitlength}{0.9cm} 
  \begin{figure}[h]
    \begin{center}
\begin{picture}(6, 6)

\qbezier(0,4)(2,5)(2,6)
\qbezier(0,2)(2,1)(2,0)

\qbezier(6.65,3.6)(2.5,5)(3,6)
\qbezier(3,0)(2,3)(6.65,2.6)
\qbezier(1.92,5.5)(2.5,5.2)(3.05,5.5)
\qbezier[20](1.92,5.5)(2.5,5.7)(3.05,5.5)


\qbezier(1.92,0.5)(2.5,0.2)(2.88,0.5)
\qbezier[20](1.92,0.5)(2.5,0.7)(2.88,0.5)


\qbezier(-0.15,3.2)(3.25,2.7)(6.65,3.25)

\put(3.2,5.5){$M_{b'}^{b'}(\varphi)$}
\put(3.2,0.5){$M_{b'}^{b'}(\varphi)$}
\put(2.5,3.2){$M_{\lambda}^{\lambda}(\varphi)$}
\end{picture}
\caption{The proof of Theorem ~\ref{Th:ends}-(A3).}\label{The proof of Theorem 1.2-(A3)}
    \end{center}
  \end{figure}
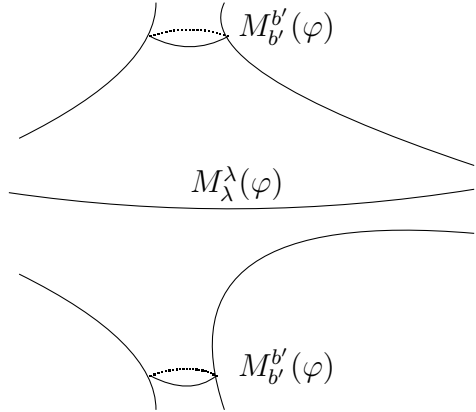

The proof of (A3) is now a straightforward consequence of (D2), see Figure \ref{The proof of Theorem 1.2-(A3)}. 

If $M_b^b(\varphi)$ is compact for some $b\in \varphi(M)$, then $\varphi^{-1}[b,\infty)$ has two ends. From the assumption and Theorem \ref{Th:disconnected} (1), we observe that $M_\lambda^\lambda(\varphi)$ is non-compact. Therefore $M^b(\varphi)=\varphi^{-1}[\lambda,b]$ is non-compact and hence contains at least one end. This proves (A3).


We prove  Theorem ~\ref{Th:ends} (E).
Suppose that $\varphi$ admits both compact and non-compact levels simultaneously. The same notation as in the proof of (D) will be used. If $\varphi$ admits a disconnected level, then $\varphi^{-1}[b',\infty)$, 
for all $b'>b$, consists of two unbounded components.

Then Theorem \ref{Th:disconnected} (1) and Lemma \ref{Lem:compactlevel} imply that $\lambda:=\inf_M\varphi$ is attained and $M^\lambda_\lambda(\varphi)$ is connected and noncompact. Therefore, every compact set $K$ containing $M^{b'}_{b'}(\varphi)$ has the property that $M\setminus K$ has more than two unbounded components. In fact, two components of $M\setminus K$ contain $\varphi^{-1}[b',\infty)$ and the other component intersects with $M^\lambda_\lambda(\varphi)$ outside $K$. This proves that $M$ has at least three ends, a contradiction to the assumption of (E).\par
If all the levels of $\varphi$ are connected and non-compact, then $M$ has one end by (C), a contradiction to the assumption of (E). This completes the proof of (E).

\end{document}